\documentclass[10pt]{article}
\usepackage{amsfonts,latexsym,eucal,amsmath,amsthm,amssymb,bm}
\usepackage{latexsym,eucal,bm}
\usepackage{graphicx}
\usepackage{array}
\usepackage[font=footnotesize]{subfig}
\usepackage{natbib}
\usepackage{enumerate}
\usepackage{cite}
\usepackage{booktabs}
\usepackage{parskip}
\newtheorem{theorem}{Theorem}[section]
\newtheorem{lemma}[theorem]{Lemma}
\newtheorem{corollary}[theorem]{Corollary}
\newtheorem{proposition}[theorem]{Proposition}

\theoremstyle{definition}
\newtheorem{definition}[theorem]{Definition}
\newtheorem{remark}[theorem]{Remark}

\numberwithin{equation}{section}

\renewcommand{\tilde}{\widetilde}
\renewcommand{\hat}{\widehat}
\renewcommand{\Pr}{\mathbb{P}}
\renewcommand{\Re}{\mathbb{R}}
\newcommand{\Ex}{\mathbb{E}}
\newcommand{\Cov}{\mathrm{cov}}

\makeatletter
\def\thm@space@setup{%
   \thm@preskip=\parskip \thm@postskip=0pt
}
 \makeatother

\usepackage[colorlinks=true,linkcolor=blue,citecolor=blue]{hyperref}
\setcitestyle{round}

\begin{document}

\title{A central limit theorem for scaled eigenvectors of random dot product graphs}
\author{Avanti Athreya, Vince Lyzinski, David J. Marchette, \\%
Carey E. Priebe, Daniel L. Sussman, Minh Tang}
\maketitle

\begin{abstract}
  We prove a central limit theorem for the components of the largest
  eigenvectors of the adjacency matrix of a finite-dimensional random dot
  product graph whose true latent positions are unknown. In
  particular, we follow the methodology outlined in
  \citet{sussman2012universally} to construct consistent estimates for
  the latent positions, and we show that the appropriately scaled
  differences between the estimated and true latent positions converge
  to a mixture of Gaussian random variables. As a corollary, we
  obtain a central limit theorem for the first eigenvector of the
  adjacency matrix of an Erd\"os-Renyi random graph.
\end{abstract}


\section{Introduction}
\label{sec:introduction}
Spectral analysis of the adjacency and Laplacian matrices for graphs
is of both theoretical \citep{chung1997spectral} and practical
\citep{von2007tutorial} significance.  For instance, the spectrum can
be used to characterize the number of connected components in a graph
and various properties of random walks on graphs, and the eigenvector
corresponding to the second smallest eigenvalue of the Laplacian is
used in the solution to a relaxed version of the min-cut problem
\citep{fiedler1973algebraic}. In our current work, we investigate the
second-order properties of the eigenvectors corresponding to the
largest eigenvalues of the adjacency matrix of a random graph. In
particular, we show that under the random dot product graph model
\citep{young2007random}, the components of the eigenvectors are
asymptotically normal and centered around the true latent positions
(see Section~\ref{sec:MultiD}). We consider only undirected,
loop-free graphs in which the expected number of edges grows as
$\Theta(n^2)$. However, the results contained here can be
extended to sparse graphs.

This paper is organized as follows: in Section~\ref{sec:background} we
provide background and give a brief overview of related work. In
Section~\ref{sec:oned} we prove a central limit theorem for the difference between the estimated and true latent positions for the one-dimensional random dot product graph.  We present the proof of the one-dimensional case first because it illustrates, in a simpler setting, the main ideas of the proof in higher dimensions.
We then note two corollaries for special cases of random dot
product graphs. In Section~\ref{sec:MultiD} we derive the central
limit theorem for multi-dimensional random dot product
graphs, and in Section~\ref{sec:sim} we demonstrate our results via
simulation. Finally we conclude the paper in Section~\ref{sec:disc}
with further discussion.

\section{Background and Related Work}
\label{sec:background}
This work is concerned with the eigenvectors corresponding to the
largest eigenvalues of the adjacency matrix of a {\em random dot
  product graph}. Random dot product graphs are a specific
example of {\em latent position random graphs} \citep{Hoff2002}, in
which each vertex is associated with a latent position and,
conditioned on the latent positions, the presence or absence of all
edges in the graph are independent. The edge presence probability is
based on {\em link} function, which is a symmetric function of the
two latent positions.

We note briefly that, in a strong sense, latent position graphs are
identical to {\em exchangeable random graphs}
\citep{aldous1981representations,hoover1979relations}, with the key
unifying ingredient being the conditional independence of the edges.
A fundamental result on exchangeable graphs is the notion of a graph
limit which is constructed via subgraph counts
\citep{diaconis2007graph}. The work of Diaconis and Janson has important consequences in
statistical inference, for instance the method of moments for subgraph
counts \citep{bickel2011method}. In a similar spirit, our current results provide
asymptotic distributions for spectral statistics that have the promise
to improve current statistical methodology for random graphs (see \S~\ref{sec:sim}).


Statistical analysis for latent position random graphs has received
much recent interest: see \citet{goldenberg2010survey} and
\citet{Fortunato2010Community} for reviews of the pertinent
literature. Some fundamental results are found in
\citet{bickel2011method, Bickel2009, Choi2010} among many others.  In
the statistical analysis of latent position random graphs, a common
strategy is to first estimate the latent positions based on some
specified link function. For example, in random dot product graphs
\citep{young2007random}, the link function is the dot product: namely,
the edge probabilities are the dot products of the latent
positions. \citet{sussman2012universally} show that spectral
decompositions of the adjacency matrix for a random dot product graphs
provide accurate estimates of the underlying latent positions. In this
work, we extend the analysis in \citet{sussman2012universally} to show
a distributional convergence of the residuals between the estimated
and true latent positions.

Our work is also influenced by the analysis of the spectra of random
graphs \citep{chung1997spectral}. Of special note is the classic paper
of \citet{furedi1981eigenvalues}, in which the authors show that for
an Erd{\"o}s-R{\'e}nyi graph with parameter $p$, the appropriately
scaled largest eigenvalue of the adjacency matrix converges in law to
a normal distribution. Other results of this type are proved for
sparse graphs in both the independent edge model
\citep{krivelevich2003largest} and the $d$-regular random graph model
\citep{janson2005first}. More recently, general bounds for the
operator norm of the difference between the adjacency matrix and its
expectation have been proved in \citet{oliveira2009concentration} and
\citet{tropp2011freedman} (see Proposition~\ref{thm:oldBnd_oned},
Eq.~\eqref{eq:opNormBnd_oned}).

We would, of course, be remiss not to mention important recent results in random matrix theory. In particular, a
recent result by \citet{tao2012random} proves a central limit theorem
for the eigenvectors of a mean zero random symmetric matrix with
independent entries. \citet{tao2012random} prove a result for
eigenvectors corresponding to the bulk of the spectra and
\citet{knowles2011eigenvector} prove a similar result for eigenvectors
near the ``edge'' of the spectra. A material difference between these
results and our present work, however, is that we consider random
matrices whose entries have nonzero mean. For mean zero matrices, the
eigenvalues and eigenvectors that are most readily studied are not the
largest in magnitude but those in the ``bulk'' of the spectra, while
in our setting, the structure of the mean matrix eases the study of
the largest eigenvalues and their corresponding eigenvectors.

As will be seen in Section~\ref{sec:oned}, the key step in our work is
to apply the power method to the adjacency matrix, with the initial
vector the true latent position. Conditioned on the true
latent position, this produces a vector whose components are
asymptotically normally distributed.  Furthermore, the difference
between this vector and the true eigenvector of the adjacency matrix
is asymptotically negligible, due to a large gap between the largest
eigenvalue and the remaining eigenvalues.

Finally, we note that a recent paper of \citet{yan13} which provides a proof of the
asymptotic normality for maximum likelihood estimates of the
parameters of a related model, i.e. the logistic $\beta$-model, and derived the
associated Fisher information matrix. The $\beta$-model
also belongs to the class of latent position models.
It is thus of potential interests to derive the Fisher information
for the random dot product graph model as considered in this work.



\section{Central limit theorem for one-dimensional random dot product graphs}
\label{sec:oned}
In this section, we state and prove a central limit theorem for a
one-dimensional random dot product graph, defined as follows.
\begin{definition}[Random Dot Product Graph ($d$=1)]
\label{def:rdpg-1dim}
  For a distribution $F$ on $[0,1]$, we say that $(X,A)\sim
  \mathrm{RDPG}(F)$ if the following hold. Let $X_1,\dotsc, X_n
  {\sim} F$ be independent random variables and define
\begin{equation}
\label{eq:defXP-1dim}
  X=[X_1,\dotsc,X_n]^\top\in \Re^{n\times 1}\text{ and }
P=XX^\top\in [0,1]^{n\times n}.
\end{equation}
The $X_i$'s are the latent positions for the random graph with adjacency matrix $A$, where
$A\in\{0,1\}^{n\times n}$ is defined to be a symmetric, hollow matrix
such that for all $i<j$, conditioned on $X_i$ and $X_j$,
\begin{equation}
 A_{ij}\stackrel{ind}{\sim}\mathrm{Bern}(X_i X_j).
\end{equation}
\end{definition}
We remark that this one-dimensional model is a slight modification of
the {\em rank 1 inhomogeneous random graph model} studied as an
example in \citet{bollobas2007phase}.

Given $A$, it is often important to estimate $X$.  Our estimate for $X$, which we denote $\hat{X}$, is defined by $\hat{X}=\hat{\lambda}^{1/2} \hat{V}$,
where $\hat{\lambda}=\lambda_1(A)$ is the largest eigenvalue of $A$
and $\hat{V}$ its associated eigenvector, normalized to be of unit
length. We define $V=X/\|X\|$, so $V$ is the normalized true
latent positions.  Let $\delta=\Ex[X_1^2]$ be the second moment of the
latent positions.

Throughout this work, we will need explicit control on the
differences, in Frobenius norm, between $X$ and $\hat{X}$ and $V$ and
$\hat{V}$. We state here the necessary bounds in the one-dimensional
case. These bounds are special cases of the more general bounds in the
finite-dimensional setting of \S~\ref{sec:MultiD}. The proofs for the
finite-dimensional bounds are given in
\citet{sussman2012universally} and \citet{oliveira2009concentration}.
\begin{proposition}
\label{thm:oldBnd_oned}
Let $\delta=\Ex[X_1^2]$,
$V= X/\|X\|$ and $\hat{V} = \hat{X}/\|\hat{X}\|$. Let $c>0$ be arbitrary. There exists a constant $n_0(c)$ such that if $n> n_0$,  then for any $\eta$ satisfying $n^{-c} < \eta <1/2$, the following bounds hold with probability greater than
$1- \eta$,
\begin{align}
\|X-\hat{X}\| &\leq 4 \delta^{-1}\sqrt{2 \log(n/\eta)},\\
 \|V-\hat{V}\| &\leq 4 \delta^{-1}\sqrt{\frac{\log(n/\eta)}{n}}, \\
\text{ and } \|XX^\top- A\| &\leq 2 \sqrt{n \log{(n/\eta)}}.\label{eq:opNormBnd_oned}
\end{align}
Hence, the above  bounds imply that for $n$ sufficiently large, with
probability greater than $1-\eta$,
\begin{equation}
 \frac{\delta n}{2}\leq \|P\| \leq n, \quad
\frac{\delta n}{2}\leq \|A\| \leq n \label{eq:SAopnorm_oned}.
\end{equation}
where $\|\cdot \|$ represents the spectral norm for matrices and the
$\ell_2$ vector norm for vectors.
\end{proposition}
We emphasize that a number of our subsequent arguments focus on events that occur with probability at least $1-\eta$; it is assumed that $c$ is suitably chosen and $n$ is sufficiently large to ensure that $n^{-c} \leq \eta \leq 1/2$.

Our aim in this section is to prove the following limit theorem.
\begin{theorem} \label{thm:clt_oned}
Let $(X,A)\sim \mathrm{RDPG}(F)$
  and let $\hat{X}$ be our estimate for $X$.  Let $\Phi(z, \sigma^2)$
  denote the normal cumulative distribution function, with mean zero
  and variance $\sigma^2$, evaluated at $z$. Then for each component
  $i$ and any $z \in \mathbb{R}$,
  \begin{equation*}
    \Pr\Bigl\{\sqrt{n}\bigl( \hat{X}_i - X_i \bigr) \leq z\Bigr\}
\rightarrow  \int \Phi(z, \delta^{-2}\sigma^2(x_i))dF(x_i)
  \end{equation*}
  where $\sigma^2(x_i) = x_i \Ex[X_1^3]-x_i^2\Ex[X_1^4]$ and
  $\delta=\Ex[X_1^2].$ That is, the sequence of random variables
  $\sqrt{n}( \hat{X}_i -X_i)$ converges in distribution to
  a mixture of normals.  We denote this mixture by $\mathcal{N}(0,
  \delta^{-2}\sigma^2(X_i))$.
\end{theorem}
As an immediate consequence, we obtain the following corollary for the
eigenvectors of an Erd{\"o}s-R{\'e}nyi random graph.  For the
Erd\"os-Renyi graph, the $X_i$ have a degenerate distribution; namely,
there is some $p\in(0,1)$ such that $X_i=\sqrt{p}$ for all $i$.
\begin{corollary}\label{cor:clt_for_ER}
  For an Erd\"os-Renyi ($p$) graph, the following central limit
  theorem holds:
\[\sqrt{n}( \hat{X}_i -\sqrt{p}) \stackrel{\mathcal{L}}{\longrightarrow} \mathcal{N}(0, 1-p).\]
\end{corollary}

To prove Theorem~\ref{thm:clt_oned}, we will need
Proposition~\ref{thm:oldBnd_oned} and a succession of simpler lemmas.
To begin, we apply one step of the power method, with initial vector
$V$.  In particular, let $Y=\hat{\lambda}^{-1/2}AV$ be the vector in
$\mathbb{R}^n$ with components $Y_i=\hat{\lambda}^{-1/2}
\sum_{j=1}^n A_{ij} V_j$.
\begin{proposition}\label{prop:conditional_CLT}
  Taking $X_i=x_i$ as given, we have
  \begin{equation*}
  \sqrt{n}\Bigl( Y_i -\frac{\|X\|}{\hat{\lambda}^{1/2}}x_i \Bigr) \stackrel{\mathcal{L}}{\longrightarrow} \mathcal{N}(0, \delta^{-2}\sigma^2(x_i))
  \end{equation*}
  where $\sigma^2(\cdot)$ and $\delta$ are as in Theorem~\ref{thm:clt_oned}.
\end{proposition}
\begin{proof}
Observe that
\begin{align*}
  \sqrt{n}\Bigl(Y_i - \frac{\|X\|}{\hat{\lambda}^{1/2}}x_i\Bigr)&=\frac{\sqrt{n}}{\hat{\lambda}^{1/2}}\Bigl(\sum_{j=1}^n A_{ij} V_j - \|X\| x_i\Bigr)\\
  &=\frac{\sqrt{n}}{\hat{\lambda}^{1/2}\|X\|}\Bigl(\sum_{j=1}^n A_{ij}X_j -x_i\|X\|^2\Bigr)\\
  &=\frac{n}{\hat{\lambda}^{1/2}\|X\|} \ast \frac{1}{\sqrt{n}} \Bigl[\Bigl(\sum_{j\neq i} (A_{ij}-x_i X_j)X_j\Bigr)-x_i^3\Bigr]
\end{align*}
The scaled sum
\begin{equation*}
\frac{1}{\sqrt{n}} \Bigl(\sum_{j\neq i} (A_{ij}-x_i X_j)X_j \Bigr)
\end{equation*}
is a sum of independent, identically distributed random variables each
with mean zero and variance
$$\sigma^2(x_i)=x_i\Ex[X_j^3]-x_i^2\Ex[X_j^4].$$
The classical Lindeberg-Feller central limit theorem and Slutsky's
Theorem \citep[ Theorem 7.2.1 and Theorem 4.4.6]{chung1974course}
imply that
\begin{equation}\label{eq:stdCLT}
\frac{1}{\sqrt{n}} \Bigl[\Bigl(\sum_{j\neq i} (A_{ij}-x_i X_j)X_j\Bigr)-x_i^3\Bigr]\stackrel{\mathcal{L}}{\longrightarrow} \mathcal{N}(0, \sigma^2(x_i)).
\end{equation}
Furthermore, the Strong Law implies that $\|X\|^2/n \rightarrow \delta$,  and \eqref{eq:opNormBnd_oned} and \eqref{eq:SAopnorm_oned} in Proposition~\ref{thm:oldBnd_oned}, along with the Borel-Cantelli Lemma,
imply that $n/(\hat{\lambda}^{1/2}\|X\|)$ converges almost surely to
$\delta^{-1}$.  Finally,  another application of Slutsky's Theorem allows us to
conclude that
\begin{equation*}
 \sqrt{n}\Bigl(Y_i -\frac{\|X\|}{\hat{\lambda}^{1/2}} x_i\Bigr) \stackrel{\mathcal{L}}{\longrightarrow} \mathcal{N}(0, \delta^{-2} \sigma^2(x_i))
\end{equation*}
as desired.
\end{proof}

We remark that Proposition~\ref{thm:oldBnd_oned} and the Borel-Cantelli Lemma imply that
$n/(\hat{\lambda}^{1/2} \|X\|) \rightarrow \delta^{-1}$ with probability one, but we will need additional control of the rate of this convergence.   We use the notation $Y_n=O_{\mathbb{P}} (\alpha(n))$ to denote that  the sequence of random variables $Y_n/\alpha(n)$ is bounded in probability; and $Y_n=o_{\mathbb{P}}(\alpha(n))$ to denote that the sequence of random variables $Y_n/\alpha(n)$ converges to zero in probability.
 The next lemma shows that the factor
$(1-\|X\|/\hat{\lambda}^{1/2})=o_{\mathbb{P}}(1/\sqrt{n})$  . As a
remark, for ease of exposition, some of the subsequent results and
proofs in this paper, e.g.,
Proposition~\ref{prop:bounds_on_Y-X}, will contain bounds with
universal but hidden constants that do
not depend on the parameters---that is, they do not depend on $n$,
$\eta$ or $\delta$. We use the convention that
these hidden constants are all denoted by a generic symbol $C$ and
can change from line to line in the paper.
\begin{lemma}
  In the setting of Theorem~\ref{thm:clt_oned}, $\sqrt{n}\Bigl( 1-
    \hat{\lambda}^{-1/2}\|X\|\Bigr)\stackrel{p}{\longrightarrow}
  0$. \label{lem:Lyz}
\end{lemma}
\begin{proof}
  First, we have
  \begin{equation}
    \label{eq:3}
    1- \frac{\|X\|}{\hat{\lambda}^{1/2}}=
    \frac{\hat{\lambda}-\|X\|^2}{\hat{\lambda}^{1/2}(\|X\| +
      \hat{\lambda}^{1/2})}.
  \end{equation}
  By Proposition~\ref{thm:oldBnd_oned}, the denominator of
  Eq.~\eqref{eq:3} is bounded by
  \begin{equation*}
   \hat{\lambda}^{1/2}(\|X\| + \hat{\lambda}^{1/2}) \geq \frac{\delta n}{2}
  \end{equation*}
  with probability at least $1-\eta$. Since
\begin{equation*}
\|X\|^2 = \lambda_1(P)= V^\top P V \textrm{ and } \hat{\lambda}=\hat{V}^\top A \hat{V},
\end{equation*}
it follows that
\begin{equation}
\Bigl|\|X\|^2 - \hat{\lambda}\Bigr| \leq  |V^\top P V - V^\top  A V| + |V^\top A V - \hat{V}^\top A \hat{V}|.  \label{eq:LyzDecomp_oned}
\end{equation}
For the second term on the right-hand-side of
\eqref{eq:LyzDecomp_oned}, first observe that
\begin{align}
 |V^\top A V - \hat{V}^\top A \hat{V}| &=  |(V-\hat{V})^\top  A (V-\hat{V}) + 2(V-\hat{V})^\top A \hat{V} |\notag \\
&= |(V-\hat{V})^\top A(V-\hat{V}) + 2\hat{\lambda }(V-\hat{V})^\top \hat{V}|\notag \\
& \leq |(V-\hat{V})^\top A(V-\hat{V})| + |2\hat{\lambda }(V-\hat{V})^\top \hat{V}|\notag \\
&\leq \hat{\lambda}\|V-\hat{V}\|^2 +  2\hat{\lambda }|(V-\hat{V})^\top \hat{V}|\notag
\end{align}
Now observe that
\begin{equation}
 \label{eq:hatVV}
  |\hat{V}^\top (V-\hat{V})|= \frac{1}{2}|\hat{V}^\top V-\hat{V}^\top
  \hat{V} + V^\top \hat{V} - V^\top V| = \frac{1}{2}\|V-\hat{V}\|^2
\end{equation}
and this implies that with probability at least $1-\eta$,
\begin{equation}
 |V^\top A V - \hat{V}^\top A \hat{V}|\leq 2 \hat{\lambda} \|V-\hat{V}\|^2  \leq \frac{C \log(n/\eta)}{\delta^2}, \label{eq:lemBnd1}
\end{equation}
for some constant $C$, by Proposition~\ref{thm:oldBnd_oned}.
For the first term in Eq.~\ref{eq:LyzDecomp_oned}, we have
\begin{equation}\label{eq:Bnd_UPU_UAU_oned}
  \begin{split}
|V^\top A V-V^\top P V| &= |\sum_{i,j} (A_{ij}-P_{ij}) V_i V_j|
\\ &\leq |2\sum_{i<j} (A_{ij} -P_{ij}) V_i V_j + \sum_{i} P_{ii}
V_i^2| \\ &\leq |2\sum_{i<j} (A_{ij} -P_{ij}) V_i V_j + \sum_{i}
V_i^2| \\ &\leq |2\sum_{i<j} (A_{ij} -P_{ij}) V_i V_j + 1|
\end{split}
\end{equation}
The term $2 \sum_{i < j} (A_{ij} - P_{ij}) V_{i} V_{j}$ in Eq.\eqref{eq:evalHoef1}
is the sum of $\binom{n}{2}$ independent
random variables and we can use
Hoeffding's inequality to conclude that
\begin{equation*}
  \begin{split}
  \Pr[ |\sum_{i<j} 2 (A_{ij}-P_{ij}) V_{i} V_{j}|\geq t ] & \leq 2\exp
 \Bigl( \frac{-2t^2}{\sum_{i<j} (2V_{i} V_{j})^2}\Bigr) \\
&\leq 2\exp
 \Bigl( \frac{-t^2}{\sum_{i=1}^{n} \sum_{j=1}^{n} (V_{i}
     V_{j})^2}\Bigr) \\
&\leq 2\exp(-t^2).
  \end{split}
\end{equation*}
where the last line follows from the fact that $V$ is a unit vector.  Therefore, this bound is independent of the choice of $V$.
Hence, the first term in Eq.\eqref{eq:LyzDecomp_oned} satisfies
\begin{equation*}
\Pr[ | V^\top(A-P)V | \leq C\sqrt{\log(1/\eta)}+1 ] \geq 1-\eta.
\end{equation*}
Putting together the bounds on the two terms in
Eq.~\eqref{eq:LyzDecomp_oned} yields the bound
\begin{equation*}
  |\|X\|^{2} - \hat{\lambda} | \leq \frac{C\log{(n/\eta)}}{\delta^2}
\end{equation*}
with probabilty at least $1 - \eta$. We therefore have $1 -
\hat{\lambda}^{-1/2}\|X\| = O_{\Pr}(1/n)$ and the
claim in Lemma~\ref{lem:Lyz} follows.
\end{proof}

\begin{remark}\label{rem:eigValBnd}
  The proof of the previous lemma shows that with high
  probability, $|\hat{\lambda}-\|X\|^2|\leq C\log n$ for some $C$
  depending only on the distribution $F$. This bound is similar in
  kind to the central limit theorem, proved in
  \citet{furedi1981eigenvalues}, for the largest eigenvalue of an
  Erd\"{o}s-R\'{e}nyi random graph. \citet{alon2002concentration} also
  provide similar concentration rates for the first eigenvalue, namely
  that $|\lambda-\Ex[\lambda]|$ can be tightly controlled, which is
  somewhat different from our result. This result also greatly
  improves on the bound one obtains using only the operator norm bound
  of \citet{oliveira2009concentration} (see
  Proposition~\ref{thm:oldBnd_oned}).
\end{remark}

Finally, we prove a bound on the $\ell_2$ distance between
$\hat{\lambda}^{-1/2}AV$ and $\hat{X}$.
\begin{proposition}\label{prop:bounds_on_Y-X}
  Let $Y= \hat{\lambda}^{-1/2}AV$. Provided the events in Theorem
  ~\ref{thm:oldBnd_oned} occur, we have
\begin{equation*}
  \|Y-\hat{X}\|  \leq \frac{C\log(n/\eta)} {\sqrt{ n\delta^3}}.
\end{equation*}
\end{proposition}
\begin{proof} Let $E=A-P=A-XX^\top$. We have
  \begin{equation*}
   \begin{split}
 \|Y-\hat{X}\| &= \hat{\lambda}^{-1/2} \|A(V-\hat{V})\| \\ &=
 \hat{\lambda}^{-1/2}\|(\hat{\lambda}
 \hat{V}\hat{V}^\top+E)(V-\hat{V})\| \\ &\leq  \hat{\lambda}^{1/2} |\hat{V}^\top(V-\hat{V})| +  \hat{\lambda}^{-1/2}\|E (V-\hat{V})\|
\end{split}
  \end{equation*}
Using Eq.~\eqref{eq:hatVV} and $\sqrt{\delta
  n}/2\leq\hat{\lambda}^{1/2}\leq 2\sqrt{\delta n}$,
the term $\hat{\lambda}^{1/2} |\hat{V}^\top(V-\hat{V})|$ is bounded above by $C
\delta^{-3/2}n^{-1/2}\log(n/\eta)$. By Proposition~\ref{thm:oldBnd_oned},
$\|E\| \leq 2 \sqrt{n\log(n/\eta) }$, and therefore
\begin{equation*}
 \begin{split}
 \hat{\lambda}^{-1/2}\|E(V-\hat{V})\| &\leq
 \hat{\lambda}^{-1/2}\|E\| \ast \|(V-\hat{V})\| \leq \frac{C
   \log{(n/\eta)}}{\sqrt{n \delta^3}}
 \end{split}
\end{equation*}
from which the desired bound follows.
\end{proof}

We are now equipped to prove our limit theorem:
\begin{proof}[Proof of Theorem~\ref{thm:clt_oned}]
  Integrating over the possible realizations of $X_i$ in
  Proposition~\ref{prop:conditional_CLT} and applying the dominated
  convergence theorem, we deduce that
  \begin{equation*}
\Pr\Bigl\{\sqrt{n}\Bigl( Y_i - \frac{\|X\|}{\hat{\lambda}^{1/2}}X_i
  \Bigr) \leq z\Bigr\} \rightarrow \int \Phi(z, \delta^{-2}\sigma^2(x_i))  dF(x_i).
  \end{equation*}
This establishes that
\begin{equation*}
 \sqrt{n}\Bigl( Y_i -\frac{\|X\|}{\hat{\lambda}^{1/2}} X_i \Bigr)
\rightarrow \mathcal{N}(0, \delta^{-2}\sigma^2(X_i)).
\end{equation*}
Markov's inequality, the exchangeability of
$\{Y_i-\hat{X}_i\}_{i=1}^n$, and the bounds in
Prop.~\ref{prop:bounds_on_Y-X} allow us to conclude that
\begin{equation*}
\Pr[\sqrt{n}|Y_i - \hat{X}_i|>\epsilon] \leq
\frac{\Ex[n(Y_i-\hat{X}_i)^2]}{\epsilon^2} =
\frac{\Ex[\|Y-\hat{X}\|^2]}{\epsilon^2}\leq \frac{C\log^2{n}}{\epsilon^2 n}.
\end{equation*}
for some constant $C$ that is independent of $n$ and $\epsilon$.
Hence $\sqrt{n}(Y_i - \hat{X}_i)$ converges to zero in
probability. Observe that
\begin{equation*}
  \sqrt{n}\Bigl(\hat{X}_i -\frac{\|X\|}{\hat{\lambda}^{1/2}} X_i
  \Bigr) = \sqrt{n}
(\hat{X}_i -Y_i) + \sqrt{n}\Bigl(Y_i- \frac{\|X\|}{\hat{\lambda}^{1/2}}X_i \Bigr),
\end{equation*}
and Theorem~\ref{thm:clt_oned} follows from the convergence to zero in
probability of the first summand, the convergence in distribution to a
Gaussian mixture of
the second summand, and Lemma~\ref{lem:Lyz}.
\end{proof}

\subsection{Corollaries}\label{sec:cor}
In this section, we prove three corollaries, each of which is either a
special case or an extension of Theorem~\ref{thm:clt_oned}. First, we
demonstrate that in the stochastic blockmodel, if we condition on
$X_i=x$, the residuals converge to the correct mixture component;
second, we prove a similar result in the case where the latent
position distribution has a density and we condition on $X_i$
belonging to any set of positive $F$-measure, where $F$ is the
distribution of the latent positions; and finally, we prove a central
limit theorem for the distribution of any fixed number $k$ of the
residuals, $\sqrt{n}(\hat{X}_i -X_i)$, for $1 \leq i \leq k$.  We
begin with the first corollary, in which we obtain appropriate
convergence to the correct mixture component for the stochastic
blockmodel, as defined in the statement of Corollary~\ref{cor:sbmCLT_1d}
below.
\begin{corollary}
  \label{cor:sbmCLT_1d}
  In the setting of Theorem~\ref{thm:clt_oned}, let
  $\mathcal{X}=\mathrm{supp}(F)\subset[0,1]$ be the support of the
  distribution of the $X_i$ and suppose that
  $|\mathcal{X}|=m<\infty$. Suppose for each $x\in \mathcal{X}$, we
  have that $\Pr[X_i=x]=\pi_x>0$. Then for all $x\in\mathcal{X}$, if
  we condition on $X_i=x$, we obtain
\begin{equation}
 \label{eq:condConv2}
  \Pr\Bigl\{ n^{1/2} (\hat{X}_i - x)\leq z \mid X_i=x \Bigr\}
  {\longrightarrow}
\Phi(z, \delta^{-2}\sigma^2(x))
\end{equation}
where $\sigma^2(\cdot)$ and $\delta$ are as in Theorem~\ref{thm:clt_oned}.
\end{corollary}

\begin{proof}
  Let $p_n(x,\epsilon) =\Pr(n^{1/2}|\hat{X_i}-Y_i|>\epsilon
    \mid X_i=x)$ where $Y=\hat{\lambda}^{-1/2}AV$ is as in
  Proposition~\ref{prop:bounds_on_Y-X}. By
  Proposition~\ref{prop:conditional_CLT} and Slutsky's Theorem, we
  need only show that for all $\epsilon>0$ and $x\in\mathcal{X}$,
\begin{equation}\label{eq:pn_to_0}
p_n(x,\epsilon)\to 0 \textrm{ as } n \to \infty,
\end{equation}
because this yields
\begin{equation}
  n^{1/2}\Bigl( \hat{X}_i -\frac{\|X\|}{\hat{\lambda}^{1/2}}
    x\Bigr)
 \stackrel{\mathcal{L}}{\longrightarrow} \mathcal{N}(0, \delta^{-2}\sigma^2(x)).
\end{equation}
First, by Markov's inequality, Proposition~\ref{prop:bounds_on_Y-X} ,
and the exchangeability of the sequence $\hat{X}_i-Y_i$, we have
\begin{equation*}
  \Pr\Bigl(n^{1/2}|\hat{X_i}-Y_i|>\epsilon\Bigr) \leq
\frac{\Ex[\|X-Y\|^2]}{\epsilon^2} \leq \frac{C \log^2 n}{n\epsilon^2}
\end{equation*}
for some $C>0$ depending only on the distribution $F$. Let $\pi_{\min}
= \min_{x\in\mathcal{X}} \pi_x$. We then have
\begin{equation*}
 \Pr\Bigl(n ^{1/2}|\hat{X_i}-Y_i|>\epsilon\Bigr) =
\sum_{x'\in \mathcal{X}} \pi_{x'} p_n(x',\epsilon) \geq  \pi_{\min} p_n(x,\epsilon)
\end{equation*}
for all $x\in \mathcal{X}$. This implies
$$p_n(x,\epsilon)\leq \frac{C \log^2 n}{n\pi_{\min}\epsilon^2},$$
which proves Eq.~\eqref{eq:pn_to_0}.
\end{proof}
This next corollary has essentially the same proof as the previous
one; here, however, the set of possible latent positions need not be
finite.  We condition, instead, on the true position belonging to a
fixed set $\mathcal{B}$ for which $P(X_i \in \mathcal{B})$ is
arbitrarily small but strictly positive.
\begin{corollary}\label{cor:cltCondSet}
  In the setting of Theorem~\ref{thm:clt_oned}, let $\mathcal{B}\subset [0,1]$ be such that
  $\Pr[X_i\in\mathcal{B}]>0$. If we condition on the event
  $\{X_i\in\mathcal{B}\}$, we obtain
\begin{equation}
\label{eq:condConv}
\Pr\Bigl\{n^{1/2}( \hat{X}_i -X_i) \leq z \mid X_i \in
  \mathcal{B} \Bigr\}
 {\longrightarrow} \frac{1}{\Pr\Bigl(X_i \in \mathcal{B}\Bigr)}
\int_\mathcal{B} \Phi(z, \delta^{-2}\sigma^2(x))dF(x_i)
\end{equation}
where $\sigma^2(\cdot)$ and $\delta$ are as in Theorem~\ref{thm:clt_oned}.
\end{corollary}

In other words, if we condition on an event of positive probability,
the convergence in distribution is to a mixture of normals where the
mixture is over only the conditioned event.  Our last corollary
asserts that our main theorem can be extended to distributional
convergence of any finite collection of estimated latent
positions. Indeed, we prove that for any finite collection
$\{i_1,\dotsc, i_{K}\}$, the residuals
$n^{1/2}(\hat{X}_{i_k}-X_{i_k})$ are asymptotically jointly normal and
asymptotically uncorrelated, and hence asymptotically independent.
\begin{corollary}\label{cor:cltK}
  Suppose $X$ and $\hat{X}$ are as in Theorem~\ref{thm:clt_oned}. Let
  $k\in\mathbb{N}$ be any fixed positive integer; let $i_1,
  \dotsc,i_K\in \mathbb{N}$ be any fixed set of indices and let $z_1,
  z_2, \cdots, z_K \in \mathbb{R}$ be fixed.  Then
\begin{equation}
\lim\limits_{n \rightarrow \infty} \Pr\Bigl[ \bigcap_{k=1}^{K} \{n^{1/2}(\hat{X}_{i_k}
-X_{i_k}) \leq z_k\}\Bigr]=
\prod_{k=1}^{K} \int_{\mathcal{X}} \Phi(z_k, \delta^{-1}\sigma^2(x_k)) dF(x_k)
\end{equation}
where $\Phi(\cdot, \sigma^2)$ denotes the cumulative distribution
function (cdf) for a normal with mean zero and variance $\sigma^2$.
Again, $\sigma^2(\cdot)$ and $\delta$ are as in
Theorem~\ref{thm:clt_oned}.

In other words, for any finite collection of indices, the residuals
between $\hat{X}$ and $X$ converge to independent mixtures of
multivariate normals which we will denote $\mathcal{N}(0,
\delta^{-2}\sigma^2(X_j))$:
\begin{equation}
\begin{pmatrix}
n^{1/2}(\hat{X}_{i_1} -X_{i_1}) \\
n^{1/2}(\hat{X}_{i_2} -X_{i_2}) \\
\vdots \\
n^{1/2}(\hat{X}_{i_K} -X_{i_K})
\end{pmatrix} \stackrel{\mathcal{L}}{\longrightarrow} \bigotimes_{k=1}^K \mathcal{N}(0, \delta^{-1}\sigma^2(X_{i_k}))\text{ as } n\to \infty.\label{eq:clt2}
\end{equation}
\label{thm:clt}
\end{corollary}

\begin{proof}[sketch]
  The proof essentially follows from an application of the
  Cram\'{e}r-Wold theorem. For ease of notation, we consider the $K=2$
  case, where we can take $i\neq i'$ and condition on $X_i=x_i$ and
  $X_{i'}=x_i$; the case for general $K$ follows similarly. This
  simplifies the covariance computation to the following:
\begin{align*}
  \Cov\Bigl\{\frac{1}{n^{1/2}} \Bigl(-x_i^3 + \sum_{j\neq i}^n
        \xi_{ij} \Bigr), \frac{1}{n^{1/2}}
    \Bigl(-x_{i'}^3 + \sum_{j'\neq i'}^n \xi_{i'j'} \Bigr)\Bigr\}
  = \frac{1}{n}\sum_{j=1}^n \sum_{j'=1}^n \Cov(\xi_{ij}, \xi_{i'j'})
\end{align*}
where $\xi_{ij} = (A_{ij} - x_i X_j) X_j$.
Now, if $j,j'\notin\{i,i'\}$, the summand is zero because the terms
are independent. This leaves only $4$ possible non-zero summands, so
the covariance is bounded by $4/n$.  The remainder of the
proof follows \emph{mutatis mutandis} as a consequence of Slutsky's
Theorem and the bounds established in
Proposition~\ref{thm:oldBnd_oned} and
Proposition~\ref{prop:bounds_on_Y-X}.
\end{proof}

\section{Central limit theorem for finite-dimensional random dot
  product graphs}
\label{sec:MultiD}

In this section, we consider general finite-dimensional random dot
product graphs and prove a central limit theorem for the scaled
differences between the estimated and true latent positions. We begin
with the construction of our estimate for the underlying latent
positions.

\begin{definition}[Random Dot Product Graph ($d$-dimensional)]
\label{def:rdpg}
  Let $F$ be a distribution on a set $\mathcal{X}\subset \Re^d$
  satisfying $\langle x, x'\rangle \in[0,1]$ for all $x,x'\in
  \mathcal{X}$. We say $(X,A)\sim \mathrm{RDPG}(F)$ if the following hold. Let
  $X_1,\dotsc, X_n {\sim} F$ be independent random variables and define
\begin{equation}
\label{eq:defXP}
X=[X_1,\dotsc,X_n]^\top\in \Re^{n\times d}\text{ and }P=XX^\top\in [0,1]^{n\times n}.
\end{equation}
The $X_i$ are the latent positions for the random graph.  The matrix
$A\in\{0,1\}^{n\times n}$ is defined to be a symmetric, hollow matrix
such that for all $i<j$, conditioned on $X_i,X_j$,
\begin{equation}
 A_{ij}\stackrel{ind}{\sim}\mathrm{Bern}(X_i^\top X_j).
\end{equation}
\end{definition}

As in Section~\ref{sec:oned}, we seek to demonstrate an asymptotically
normal estimate of $X$, the matrix of latent positions $X_1,\dotsc,
X_n$. However, the model as specified above is
non-identifiable: if $W\in \Re^{d\times d}$ is orthogonal, then $XW$
generates the same distribution over adjacency matrices. As a result,
we will often consider {\em uncentered} principal components (UPCA) of
$X$.

\begin{definition}[UPCA]
  Let $X$ and $P$ be as in Definition~\ref{def:rdpg}. Then $P$ is
  symmetric and positive semidefinite and has rank at most $d$. Hence,
  $P$ has an spectral decomposition $P=V S V^\top$
  where $V\in \Re^{n\times d}$ has orthonormal columns and $S$ is
  diagonal with positive decreasing entries along the diagonal.  The
  UPCA of $X$ is then $V S^{1/2}$ and $V S^{1/2}= X
  W_n$ for some random orthogonal matrix $W_n\in \Re^{d\times d}$.  We
  will denote the UPCA of $X$ as $\tilde{X}$.
\label{def:upca}
\end{definition}

\begin{remark}
  We denote the second moment matrix for $X_i$ by $\Delta
  =\Ex[X_iX_i^\top]$. We assume, without loss of generality, that
  $\Delta =\mathrm{diag}(\delta_1,\dotsc,\delta_d)$, i.e.,
  $\Delta$ is diagonal. For the remainder of this work we will assume
  that the eigenvalues of $\Delta$ are distinct and positive, so that
  $\Delta$ has a strictly decreasing diagonal. This is a
  mild restriction, and we impose it for technical reasons; namely, it is sufficient to ensure the requisite bounds in Lemma~\ref{lem:VThatV} below.
\end{remark}

Our estimate for $X$, or specifically our estimate for the UPCA of
$X$, is a spectral embedding defined below, which is
once again motivated by the observation that $A$ is essentially a noisy version
of $P$.

\begin{definition}[Embedding of $A$]
  Suppose that $A$ is as in Definition~\ref{def:rdpg}. Let
  $A= U_A S_A U_A^\top$ be the (full) spectral
  decomposition of $A$. Then our estimate for the UPCA of $X$ is
  $\hat{X}=\hat{V} \hat{S}^{1/2}$, where $\hat{S}\in\Re^{d\times d}$ is the
  diagonal matrix with the $d$ largest eigenvalues (in magnitude) of
  $A$ and $\hat{V}\in \Re^{n\times d }$ is the matrix with orthonormal
  columns of the corresponding eigenvectors.
\label{def:emb}
\end{definition}

Similar to the one-dimensional case, we will again need explicit control on
the differences, in Frobenius norm, between $\tilde{X}$ and $\hat{X}$,
as well as $\hat{V}$ and $V$.  These are the finite-dimensional
analogues of the bounds in Proposition~\ref{thm:oldBnd_oned} and are
proven in \citet{sussman2012universally} and
\citet{oliveira2009concentration}.
\begin{proposition}
\label{thm:oldBnd}
Suppose $X$, $A$, $\tilde{X}$ and $\hat{X}$ are as defined
above. Recall that $\delta_d$ denotes the smallest eigenvalue of
$\Delta=\Ex[X_1X_1^\top]$ . Let $c$ be arbitrary. There exists a constant $n_0(c)$ such that if $n> n_0$,  then for any $\eta$ satisfying $n^{-c} < \eta <1/2$, the following bounds hold with probability greater than
$1- \eta$,
\begin{align}
  \label{eq:6}
\|\hat{X}-\tilde{X}\|_F = \|\hat{V} \hat{S}^{1/2}-V S^{1/2}\|_F &
\leq  4 \delta_{d}^{-1} \sqrt{ 2 d \log (n/\eta)},\\
\label{eq:7}
 \|\hat{V}-V\|_F &\leq 4 \delta_{d}^{-1} \sqrt{\frac{2 d \log(n/\eta)}{n}},
\\
\text{ and } \|XX^\top- A\| &\leq 2 \sqrt{n \log(n/\eta)}. \label{eq:opNormBnd}
\end{align}
Hence, for $n$ sufficiently large, the events above imply that
with probability greater than $1-\eta$,
\begin{equation}
 \frac{\delta_d n}{2}\leq \|S\| \leq n \text{ and }
\frac{\delta_d n}{2}\leq \|\hat{S}\| \leq n. \label{eq:SAopnorm}
\end{equation}
\end{proposition}

We also need the following extension to Lemma~\ref{lem:Lyz} for
bounding the difference between the $d$ largest eigenvalues of $A$ and
the corresponding eigenvalues of $P$.
\begin{theorem}\label{thm:eigvalBnd}
  In the setting of Proposition~\ref{thm:oldBnd}, with probability
  greater than $1-2\eta$
\begin{align*}
	\|S- \hat{S}\|_F &\leq C \delta_{d}^{-2} d \log(n/\eta)
	\end{align*}
\end{theorem}

\begin{proof}
  First, consider that we can
  write $\hat{S}=\hat{V}^\top A \hat{V}$ and $S=V^\top P V$. We therefore have
\begin{equation}
\label{eq:hatSSbnd}
	\|S- \hat{S}\|_F \leq \|V^\top P V- \mathrm{diag}(V^\top A
    V)\|_F + \|\mathrm{diag}(V^\top A V)- \hat{V}^\top A \hat{V}\|_F,
\end{equation}
where $\mathrm{diag}$ denotes the operation of making off-diagonal
elements zero. As $\mathrm{diag}(\hat{V}^\top A V) = \mathrm{diag}(V^{\top} A
\hat{V})$, for the second term in the right hand side of Eq.~\eqref{eq:hatSSbnd}, we have
\begin{equation*}
 \begin{split}
\|\mathrm{diag}(V^\top A V)- \hat{V}^\top A \hat{V}\|_F &=
\|\mathrm{diag}(V^\top A V- \hat{V}^\top A \hat{V})\|_F \\
&= \Bigl\|\mathrm{diag}((V - \hat{V})^\top A (V - \hat{V}) -
2(\hat{V} - V)^\top A \hat{V}) \Bigr\|_F  \\
&\leq \|\mathrm{diag} \bigl((V - \hat{V})^{T} A (V - \hat{V}) \bigr) \|_{F} +
2 \|\mathrm{diag}\bigl((\hat{V} - V)^\top A \hat{V}\bigr) \|_F \\
&\leq \|\hat{S}\| \ast \|V-\hat{V}\|_F^2 +  2 \| \mathrm{diag}\bigl((\hat{V} -
V)^\top \hat{V} \hat{S} \bigr) \|_{F} \\
& \leq \|\hat{S}\| \ast \|V-\hat{V}\|_F^2 + 2 \|\hat{S} \| \ast \| \mathrm{diag}\bigl((\hat{V} -
V)^\top \hat{V} \bigr) \|_{F} \\
\end{split}
\end{equation*}
Furthermore, we have
\begin{align}
\mathrm{diag}(\hat{V}^\top(V-\hat{V}))
&= \frac{1}{2}\mathrm{diag}\Bigl( \hat{V}^\top V -\hat{V}^\top \hat{V}-V^\top V+V^\top \hat{V}\Bigr) \notag \\
&= \frac{1}{2} \mathrm{diag}\Bigl( (V-\hat{V})^\top (V-\hat{V}) \Bigr) \label{eq:VThatV1}
\end{align}
and this thus implies
\begin{equation*}
  \begin{split}
 \|\mathrm{diag}(V^\top A V)- \hat{V}^\top A \hat{V}\|_F & \leq \|\hat{S}\| \ast \|V-\hat{V}\|_F^2 + \|\hat{S} \| \ast \| \mathrm{diag}\bigl((\hat{V} -
V)^\top \hat{V} \bigr) \|_{F} \\ &
\leq 2 \|\hat{S} \| \ast \|\hat{V} - V \|_{F}^{2} \\
& \leq C \delta_{d}^{-2} d \log(n/\eta).
\end{split}
\end{equation*}
with probability at least $1 - \eta$.

Now, let us denote by $V_{\cdot k}$ the $k$th column of $V$ and by $V_{k \cdot}$ the $k$th row of $V$.  We point out that matrix operations (such as transposition and inversion) are assumed to be done first, and the indexing of rows or columns last.  Thus, for example, $V_{k \cdot}^{T}$ represents  the $k$th row of $V^{T}$.
Now, for the first term in the right hand side of Eq.~\eqref{eq:hatSSbnd}, we have
\begin{equation}
\begin{split}
  \|\mathrm{diag}(V^\top(A-P) V)\|_F^{2}
  &= \sum_{k=1}^{d} \Bigl(V_ {k\cdot }^{T} (A-P) V_{\cdot k}\Bigr)^2 \\
  &= \sum_{k=1}^{d} \Bigl(\sum_{i,j} (A_{ij}-P_{ij}) V_{ik} V_{jk}\Bigr)^2  \\
  &= \sum_{k=1}^{d} \Bigl(\sum_{i < j}  2 (A_{ij} - P_{ij}) V_{ik}
  V_{jk} - \sum_{i=1}^{n} P_{ii} V_{ik}^2\Bigr)^2
\end{split}
\end{equation}
We thus have
\begin{equation}
  \label{eq:evalHoef1}
  \begin{split}
   \|\mathrm{diag}(V^\top(A-P) V)\|_F
   & \leq \sum_{k=1}^{d} \Bigl| \sum_{i < j} 2 (A_{ij} - P_{ij}) V_{ik}
  V_{jk} - \sum_{i} P_{ii} V_{ik}^{2} \Bigr| \\ &\leq \sum_{k=1}^{d} \Bigl( \Bigl| \sum_{i < j} 2 (A_{ij} - P_{ij}) V_{ik}
  V_{jk} \Bigr| + 1 \Bigr)
  \end{split}
\end{equation}
as $\sum_{i} P_{ii} V_{ik}^{2} \leq 1$ because $V$ is
orthogonal and the entries of $P_{ii}$ are in $[0,1]$.
The term $\sum_{i < j} (A_{ij} - P_{ij}) V_{ik} V_{jk}$ in Eq.\eqref{eq:evalHoef1}
is once again the sum of $\binom{n}{2}$ independent
random variables and we have, by Hoeffding's inequality, that
\begin{equation*}
  \begin{split}
  \Pr[ |\sum_{i<j} 2 (A_{ij}-P_{ij}) V_{ik} V_{jk}|\geq t ] & \leq 2\exp
 \Bigl( \frac{-2t^2}{\sum_{i<j} (2V_{ik} V_{jk})^2}\Bigr) \\
&\leq 2\exp
 \Bigl( \frac{-t^2}{\sum_{i=1}^{n} \sum_{j=1}^{n} (V_{ik}
     V_{jk})^2}\Bigr) \\
&\leq 2\exp(-t^2).
  \end{split}
\end{equation*}
Hence, the first term in Eq.~\eqref{eq:hatSSbnd}  satisfies
\begin{equation*}
\Pr[ \| \mathrm{diag}(V^\top(A-P)V)\|_F \leq d(\sqrt{\log(2d/\eta)}+1) ] \geq 1-\eta.
\end{equation*}
Putting together the bounds on the two terms in
Eq.~\eqref{eq:hatSSbnd} yields the result.
\end{proof}

In the following lemma, we prove that $V^\top\hat{V}$ is
very close to the identity matrix. This extends the bound in
Eq.~\eqref{eq:hatVV} to the $d$-dimensional setting.
\begin{lemma}
\label{lem:VThatV}
In the setting of Proposition~\ref{thm:oldBnd}, with probability
greater than $1 - 2\eta$,
\begin{equation}
  \|V^\top \hat{V}-I\|_{F} \leq \frac{Cd \log(n/\eta)}{\delta_{d}^{2} n}.
\end{equation}
\end{lemma}
\begin{proof}
  The diagonal entries of $V^\top \hat{V} - I$ can be bounded by
  Eq.~(\ref{eq:VThatV1}) to yield
\begin{equation*}
 \|\mathrm{diag}((V-\hat{V})\hat{V})\|_F \leq
 \frac{1}{2}\|V-\hat{V}\|_F^2\leq \frac{Cd \log(n/\eta)}{\delta_{d}^{2} n}
\end{equation*}
with probability at least $1 - \eta$.  To bound the off-diagonal
terms, we adapt a proof from \citet{sarkar13:_role_spect} to
this somewhat different case.  First, $V^\top (A-P) \hat{V} = V^\top
\hat{V}\hat{S} - SV^\top \hat{V}$. The $ij$th entry of $V^{\top} (A -
P) \hat{V}$ can be written as
\begin{equation}
 \label{eq:VThatV2}
  V_{i\cdot }^\top(A-P) \hat{V}_{\cdot j} = (S_{ii}-\hat{S}_{jj})
  V_{i \cdot }^\top \hat{V}_{\cdot j}.
\end{equation}
Because the eigenvalues are distinct and $\|A-P\| \leq
2\sqrt{n \log(n/\eta)}$ with probability at least $1 - \eta$,
we know for $i\neq j$ that
\begin{equation*}
   |S_{ii}-\hat{S}_{jj}| \geq |S_{ii}-S_{jj}|-\|A-P\| \geq
   \delta_{d} n - 2\sqrt{n \log(n/\eta)} \geq \delta_{d} n/2
\end{equation*}
for sufficiently large $n$ with probability at least $1 - \eta$.
We also have that
\begin{equation*}
   V_{i \cdot}^\top(A-P) \hat{V}_{\cdot j} =
V_{i \cdot }^\top(A-P)V_{\cdot j} + V_{ i\cdot }^\top(A-P)(V_{\cdot j} -\hat{V}_{\cdot j})
\end{equation*}
A similar argument to the one following Eq.~\eqref{eq:evalHoef1} shows
that $|V_{i \cdot }^\top(A-P)V_{\cdot j}|\leq C \sqrt{\log (d/\eta)}$
and an application of the the Cauchy-Schwarz inequality and
Eq.~\eqref{eq:7} yields
\begin{equation*}
   |V_{i \cdot }^\top(A-P)(V_{\cdot j} -\hat{V}_{\cdot j})| \leq (C
   \sqrt{n \log(n/\eta)})(\delta_{d}^{-1} n^{-1/2}\sqrt{\log(n/\eta)})
\end{equation*}
Dividing through by $S_{ii}-\hat{S}_{jj}$ in Eq.~\ref{eq:VThatV2} gives
\begin{equation*}
  \begin{split}
 V_{i \cdot}^\top \hat{V}_{\cdot j} = \frac{V_{i \cdot}^\top(A-P) \hat{V}_{\cdot j}}{(S_{ii}-\hat{S}_{jj})}
&\leq \frac{C (\sqrt{\log (d/\eta)} + \delta_{d}^{-1}\log(n/\eta))} {\delta_{d} n/2}
\leq \frac{C\log(n/\eta)}{\delta_{d}^2 n}
  \end{split}
\end{equation*}
Eq.~\eqref{eq:evalHoef1} follows from this and the bound for the diagonal terms.\end{proof}

We now establish the central limit theorem for the scaled differences
between the estimated and true latent positions in the
finite-dimensional random dot product graph setting.
\begin{theorem}
  \label{thm:clt_multid}
  Let $(X,A)\sim \mathrm{RDPG}(F)$ be a $d$-dimensional random dot
  product graph, i.e.,
  $F$ is a distribution for points in $\mathbb{R}^{d}$, and let
  $\hat{X}$ be our estimate for $X$.  Let $\Phi(z, \Sigma)$ denote the
  cumulative distribution function for the multivariate normal, with
  mean zero and covariance matrix $\Sigma$, evaluated at $z$. Then
  there exists a sequence of orthogonal matrices $W_n$ converging to the identity almost surely such that for
  each component $i$ and any $z \in \mathbb{R}^{d}$,
  \begin{equation}
    \label{eq:4}
    \Pr\Bigl\{n^{1/2}( W_n \hat{X}_i - X_i ) \leq z\Bigr\}
\rightarrow  \int \Phi(z_i, \Sigma(x_i)) dF(x_i)
  \end{equation}
  where $\Sigma(x) =\Delta^{-1}\Ex[X_j X_j^\top(x^\top X_j -(x^\top
  X_j)^2)]\Delta^{-1}$ and $\Delta=\Ex[X_1 X_{1}^{T}]$ is the second
  moment matrix. That is, the sequence of random variables
  $n^{1/2}( W_n \hat{X}_i -X_i)$ converges in distribution to
  a mixture of multivariate normals.  We denote this mixture by
  $\mathcal{N}(0, \Sigma(X_i))$.
\end{theorem}
\begin{proof}
  The proof of the finite-dimensional central limit theorem proceeds
  in an almost identical manner to that of the one-dimensional central limit
  theorem as stated in Theorem~\ref{thm:clt_oned}. We recall
  Definition~\ref{def:upca} on the UPCA $\tilde{X}$ of $X$, i.e.,
  $\tilde{X} W_n = X$ some orthogonal matrix $W_n$. That is, for a
  given $n$, there exists an orthogonal matrix $W_n$ such that $W_n
  \tilde{X}_i = X_i$ for all $i=1,\dotsc,n$ where we recall that $X =
  [X_1, \dots, X_n]^{T}$ and $\tilde{X} = [\tilde{X}_1, \dots,
  \tilde{X}_n]^{T}$. Hence we shall prove
  Theorem~\ref{thm:clt_multid} by showing that the right hand side of
  Eq.~\eqref{eq:4} holds for the scaled difference $n^{1/2}(\hat{X}_i
  - \tilde{X}_i)$.

Let $AVS^{-1/2}=[Y_1, \cdots, Y_n]^{\top}$ where $Y_i \in \mathbb{R}^d$. We first
  show that $n^{1/2}(\tilde{X}_i - Y_i)$ is multivariate normal in the
  limit. We have
  \begin{equation*}
    \begin{split}
     n^{1/2}(\tilde{X}_i - Y_i) &=
     n^{1/2}((P\tilde{X}S^{-1})_{\cdot i}^{T} - (A\tilde{X}S^{-1})_{\cdot i}^{T})
     \\ &= n^{1/2}((P X W_n^{T} S^{-1} - A X W_n^{T} S^{-1})_{\cdot i}^{T})
    \\ &= n^{1/2} S^{-1} W_n ((PX)_{\cdot i}^{T} - (AX)_{\cdot i}^{T})  \\
    &= n S^{-1} W_n \Bigl(\frac{1}{n^{1/2}} \sum_{j=1}^{n} (P_{ij} -
    A_{ij}) X_{j}\Bigr) \\
    &= n S^{-1} W_n \Bigl( \frac{1}{n^{1/2}} \sum_{j\not=i} (A_{ij} -
        X_{i}^{T} X_j) X_j - \frac{X_i^{T} X_i}{n^{1/2}} X_i \Bigr)
    \end{split}
  \end{equation*}
Conditional on $X_i=x_i$,   the scaled sum
  \begin{equation*}
    \frac{1}{n^{1/2}} \sum_{j\not=i} (A_{ij} -
        X_{i}^{T} X_j) X_j
  \end{equation*}
  is once again a sum of i.i.d random
  variables, each with mean $0$ and covariance matrix
  $\tilde{\Sigma}(x_i) = \Ex[X_j X_j^\top(x_i^\top X_j -(x_i^\top
  X_j)^2)]$. Therefore, by the classical multivariate central limit theorem, we have
  \begin{equation*}
    \Bigl( \frac{1}{n^{1/2}} \sum_{j\not=i} (A_{ij} -
        x_{i}^{T} X_j) X_j - \frac{x_i^{T} x_i}{n^{1/2}} x_i \Bigr)
        \overset{\mathcal{L}}{\rightarrow} \mathcal{N}(0, \tilde{\Sigma}(x_i))
  \end{equation*}
 The strong law of large numbers ensures that $nS^{-1} \rightarrow \Delta^{-1} = (\mathbb{E}[X_1
  X_1^{T}])^{-1}$ almost surely. In addition,
  \begin{gather*}
 \frac{1}{n} X^{T} X = (\tilde{X} W_n^{T})^{T} (\tilde{X} W_n^{T}) W_n
 \tilde{X}^{T} \tilde{X} W_n^{T} = W_n (S/n) W_n^{T} \rightarrow W_n \Delta W_n^{T}
  \end{gather*}
  almost surely. We thus have $W_n \rightarrow I$ almost
  surely. Hence, by Slutsky's theorem in
  the multivariate setting, we have, conditional on $X_i = x_i$, that
  \begin{equation*}
    n^{1/2}(\tilde{X}_i - Y_i) \overset{\mathcal{L}}{\rightarrow}
    \mathcal{N}(0, \Delta^{-1} \tilde{\Sigma}(x_i) \Delta^{-1}) =
    \mathcal{N}(0, \Sigma(x_i))
  \end{equation*}

  Now let $\tilde{Y}_i = \hat{S}^{-1/2} (AV)_{\cdot i}^{T}$
  (contrast this with $Y_i = S^{-1/2} (AV)_{\cdot i}^{T}$. We then
  have
  \begin{equation*}
    \begin{split}
    n^{1/2} \|\tilde{Y}_i - Y_i\| &= \|n^{1/2} (S^{-1/2} \hat{S}^{-1/2} - S^{-1})
    (A\tilde{X})_{\cdot i}^{T} \| \\ &= \|n^{1/2} S^{-1/2} (\hat{S}^{-1/2} - S^{-1/2})
    W_n (A X)_{\cdot i}^{T} \| \\ &\leq n^{1/2} \|S^{-1/2} \| \ast \|A\| \ast
    \| \hat{S}^{-1/2} - S^{-1/2} \|
    \end{split}
  \end{equation*}
  We now use the notation
  $Z = \tilde{O}_{\Pr}(f(n))$ to denote that $Z$ is, with high
  probability, bounded by $f(n)$ times some multiplicative factor that
  does not depend on $n$.   Since $\|\hat{S} - S\| =
  \tilde{O}_{\Pr}(\log{n})$, i.e., $\|\hat{S} - S\|$ is bounded, with
  high probability, by a logarithmic function of $n$
  times some factor depending only
  on $\delta_{d}$ and $d$, we have $\|\hat{S}^{-1/2}
  - S^{-1/2}\| = \tilde{O}_{\Pr}(n^{-3/2} \log{n})$. We thus have
  \begin{equation*}
    \sqrt{n} \|\tilde{Y}_i - Y_i\| = \sqrt{n} \|S^{-1/2}\| \ast \|A\| \ast
    \| \hat{S} - S \| = \tilde{O}_{\Pr}(n^{-1/2} \log{n})
  \end{equation*}
  That is to say, $\sqrt{n} (\tilde{Y}_i - Y_i)$ converges to $0$ in
  probability.

  Finally, we derive a bound for $\sqrt{n}(\tilde{Y}_i -
  \hat{X}_i)$. By Markov's inequality
  \begin{equation*}
    \mathbb{P}[\sqrt{n}\|\tilde{Y}_i - \hat{X}_i\| \geq \epsilon] \leq
    \frac{\mathbb{E}[n\|\tilde{Y}_i - \hat{X}_i\|^2]}{\epsilon^2}
    = \frac{\mathbb{E}[\|AV\hat{S}^{-1/2} -
      A\hat{V}\hat{S}^{-1/2}\|_{F}^{2}]}{\epsilon^2}.
  \end{equation*}

  Let $E = A - \hat{V} \hat{S} \hat{V}^{T}$. We then have
  \begin{equation*}
    \begin{split}
    \|A \hat{V} \hat{S}^{-1/2} - A V \hat{S}^{-1/2} \|_{F} &=  \| (\hat{V}
    \hat{S} \hat{V}^{T} + E) (\hat{V} - V) \hat{S}^{-1/2} \|_{F} \\ &\leq \|
    \hat{V} \hat{S} \hat{V}^{T} (\hat{V} - V) \hat{S}^{-1/2} \|_{F} + \|E
    (\hat{V} - V) \hat{S}^{-1/2} \|_{F} \\ &
    \leq \|\hat{S}\| \|\hat{V}^{T} V - I\|_{F} \|\hat{S}^{-1/2} \| +
    \|E\| \|\hat{V} - V\|_{F} \|\hat{S}^{-1/2}\|
  \end{split}
  \end{equation*}
  By Lemma~\ref{lem:VThatV} and bounds for the spectral norm of
  $S$ and $E$, we have, with probability at least $1 - 2\eta$, that
  \begin{equation*}
    \|A \hat{V} \hat{S}^{-1/2} - A V \hat{S}^{-1/2} \|_{F} \leq C d
    \delta_{d}^{-5/2} n^{-1/2} \log{(n/\eta)}
  \end{equation*}
  On the other hand, $\|A \hat{V} \hat{S}^{-1/2} - A V \hat{S}^{-1/2}
  \|_{F}$ is at most of order $\sqrt{n}$. As $\eta$ is arbitrary, we thus have,
  \begin{equation*}
\frac{\mathbb{E}[\|AV\hat{S}^{-1/2} -
      A\hat{V}\hat{S}^{-1/2}\|_{F}^{2}]}{\epsilon^2} \leq \frac{C
      \log^{2}{n}}{\epsilon^{2} \delta_{d}^{5} n}
  \end{equation*}
  which converges to $0$ for any fixed $\epsilon > 0$ as $n
  \rightarrow \infty$. Hence, $\sqrt{n}(\tilde{Y}_i - \hat{X}_i)$
  also converges to $0$ in probability. The above reasoning can now be
  combined to show that, conditional on $X_i = x_i$,
  \begin{equation}
    \label{eq:5}
    \begin{split}
    \sqrt{n}(\hat{X}_i - \tilde{X}_i) &= \sqrt{n}(Y_i -
    \tilde{X}_i) + \sqrt{n}(\tilde{Y}_i - Y_i)
     + \sqrt{n}(\hat{X}_i - \tilde{Y}_i) \\ &\overset{\mathcal{L}}\rightarrow
    \mathcal{N}(0, \Sigma(x_i)) + o_{\Pr}(1) + o_{\Pr}(1) \\ &\overset{\mathcal{L}}\rightarrow
    \mathcal{N}(0, \Sigma(x_i)).
    \end{split}
  \end{equation}
  Finally, by an application of Slutsky's theorem to Eq.~\eqref{eq:5},
  we have, conditional on $X_i = x_i$, that
  \begin{equation*}
    \sqrt{n}(W_n \hat{X}_i
    - X_i) = \sqrt{n}(W_n \hat{X}_i - W_n \tilde{X}_i) = W_n
    \sqrt{n}(\hat{X}_i - \tilde{X}_i) \rightarrow \mathcal{N}(0, \Sigma(x_i)).
  \end{equation*}
  Eq.~\eqref{eq:4} then follows by integrating the above display over
  all the possible realizations of $X_i$ and applying the dominated
  convergence theorem.
\end{proof}

We now state the finite-dimensional analogues for the corollaries of
Section~\ref{sec:cor}. Their proofs follow directly from
Theorem~\ref{thm:clt_multid} in similar manners to their
one-dimensional counterparts.
\begin{corollary}
  \label{cor:sbmCLT_multid}
  In the setting of Theorem~\ref{thm:clt_multid}, let
  $\mathcal{X}=\mathrm{supp}(F)\subset[0,1]^{d}$ be the support of the
  distribution of the $X_i$ and suppose that
  $|\mathcal{X}|=m<\infty$. Suppose for each $x\in \mathcal{X}$, we
  have that $\Pr[X_i=x]=\pi_x>0$. Then for all $x\in\mathcal{X}$, if
  we condition on $X_i=x$, we obtain
\begin{equation}
 \label{eq:condConv2_multid}
  \Pr\Bigl\{n^{1/2}( \hat{X}_i - x)\leq z \mid X_i=x \Bigr\}
  {\longrightarrow}
\Phi(z, \Sigma(x))
\end{equation}
where $\Sigma(x)$ is as in Theorem~\ref{thm:clt_multid}.
\end{corollary}

\begin{corollary}\label{cor:cltCondSet_multid}
  In the setting of Theorem~\ref{thm:clt_multid}, suppose that
  $\mathcal{B}\subset [0,1]^{d}$ is such that
  $\Pr[X_i\in\mathcal{B}]>0$. If we condition on the event
  $\{X_i\in\mathcal{B}\}$, we obtain
\begin{equation}
\label{eq:condConv_multid}
\Pr\Bigl\{n^{1/2}( \hat{X}_i -X_i)\leq z \mid X_i \in
  \mathcal{B} \Bigr\}
 {\longrightarrow} \frac{1}{\Pr\Bigl(X_i \in \mathcal{B}\Bigr)}
\int_\mathcal{B} \Phi(z, \Sigma(x))dF(x)
\end{equation}
where $\Sigma(x)$ is as in Theorem~\ref{thm:clt_multid}.
\end{corollary}

\begin{corollary}
\label{cor:cltK_multid}
  Suppose $X$ and $\hat{X}$ are as in Theorem~\ref{thm:clt_multid}. Let
  $K\in\mathbb{N}$ be any fixed positive integer; let $i_1,
  \dotsc,i_K\in \mathbb{N}$ be any fixed set of indices and let $z_1,
  \cdots, z_K \in \mathbb{R}^{d}$ be fixed.  Then
\begin{equation}
\lim\limits_{n \rightarrow \infty} \Pr\Bigl[ \bigcap_{k=1}^{K} \{\sqrt{n}(\hat{X}_{i_k}
-X_{i_k}) \leq z_k\}\Bigr]=
\prod_{k=1}^{K} \int_{\mathcal{X}} \Phi(z_k, \Sigma(x_k)) dF(x_k)
\end{equation}
where $\Phi(\cdot, \Sigma)$ denotes the cumulative distribution
function (cdf) for a $d$-variate normal with mean zero and covariance
matrix $\Sigma$. Again the covariance matrices $\Sigma(x)$ are as in
Theorem~\ref{thm:clt_multid}.
\end{corollary}

\section{Simulations}
\label{sec:sim}
To illustrate Theorem~\ref{thm:clt_multid}, we
consider random graphs generated according to a
stochastic block model with parameters
\begin{equation}
  \label{eq:1}
  B = \begin{bmatrix} 0.42 & 0.42 \\ 0.42 & 0.5 \end{bmatrix}
  \quad \text{and} \quad \pi = (0.6,0.4).
\end{equation}
In this model, each node is either in block 1 (with probability 0.6)
or block 2 (with probability 0.4). Adjacency probabilities are
determined by the entries in $B$ based on the block memberships of the
incident vertices.  The above stochastic blockmodel corresponds to a random dot product
graph model in $\mathbb{R}^{2}$ where the distribution $F$ of the latent
positions is a mixture of point masses located at $x_1\approx (0.63,
-0.14)$ (with prior probability $0.6$) and $x_2\approx (0.69, 0.13)$
(with prior probability $0.4$).

We sample an adjacency matrix $A$ for graphs on $n$ vertices from the
above model for various choices of $n$.  For each graph $G$, let
$\hat{X}\in\Re^{n\times 2}$ denote the embedding of $A$ and let
$\hat{X}_i$ denote the $i$th row of $\hat{X}$. In
Figure~\ref{fig:clusplot}, we plot the $n$ rows of $\hat{X}$ for the
various choice of $n$.  The points are colored according to the block
membership of the corresponding vertex in the stochastic blockmodel. The ellipses show
the 95\% level curves for the distribution of $\hat{X}_i$ for each
block as specified by the limiting distribution, namely
the ellipse such that $\Pr[\hat{X}_i\in \mathrm{Ellipse}\
k|X_i=x_k]=.95$ for $k=1,2$.

\begin{figure*}[!htbp]
  \centering
  \subfloat[$n = 1000$]{
    \includegraphics[width=5.5cm]{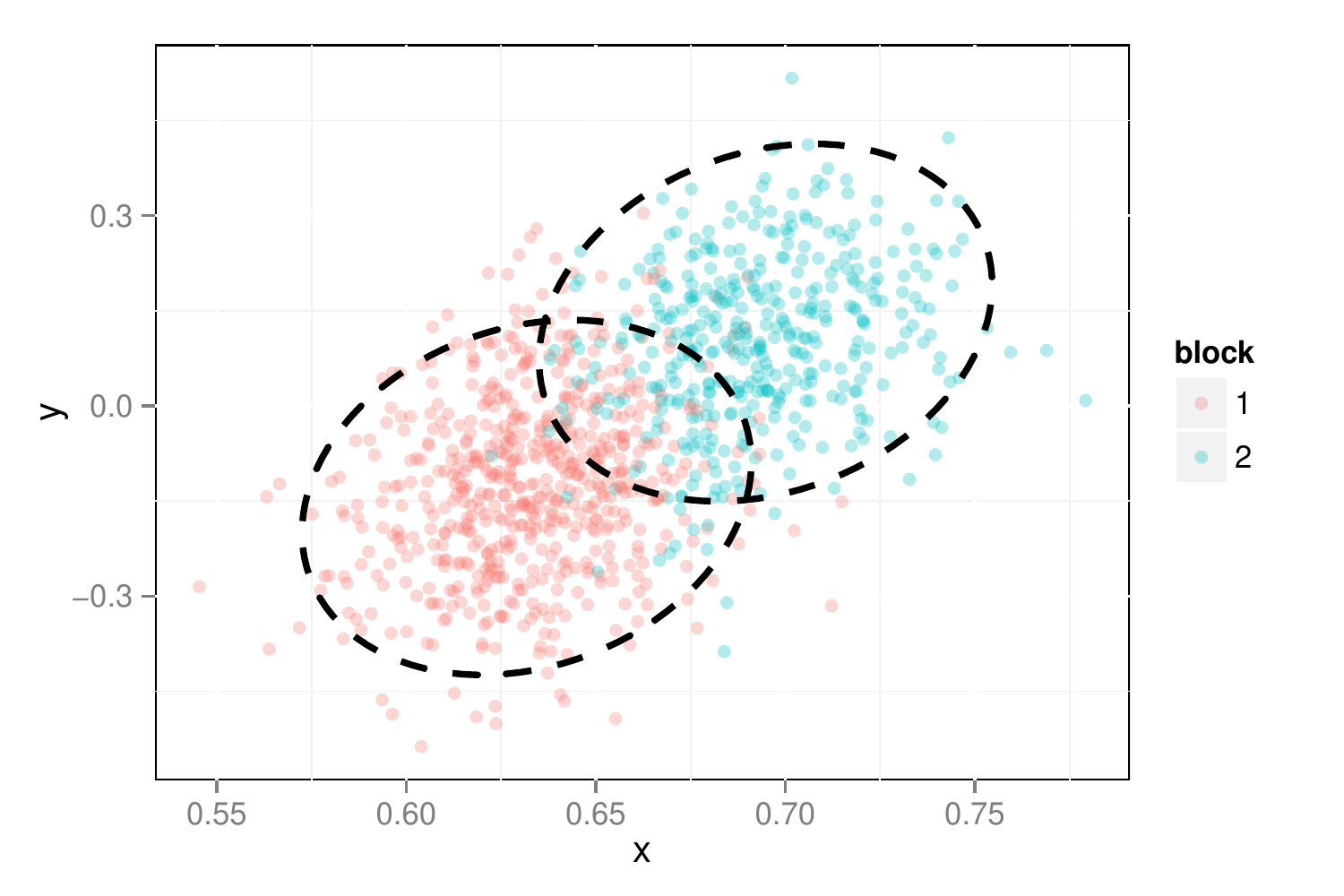}
  }
  \hfil
  \subfloat[$n = 2000$]{
    \includegraphics[width=5.5cm]{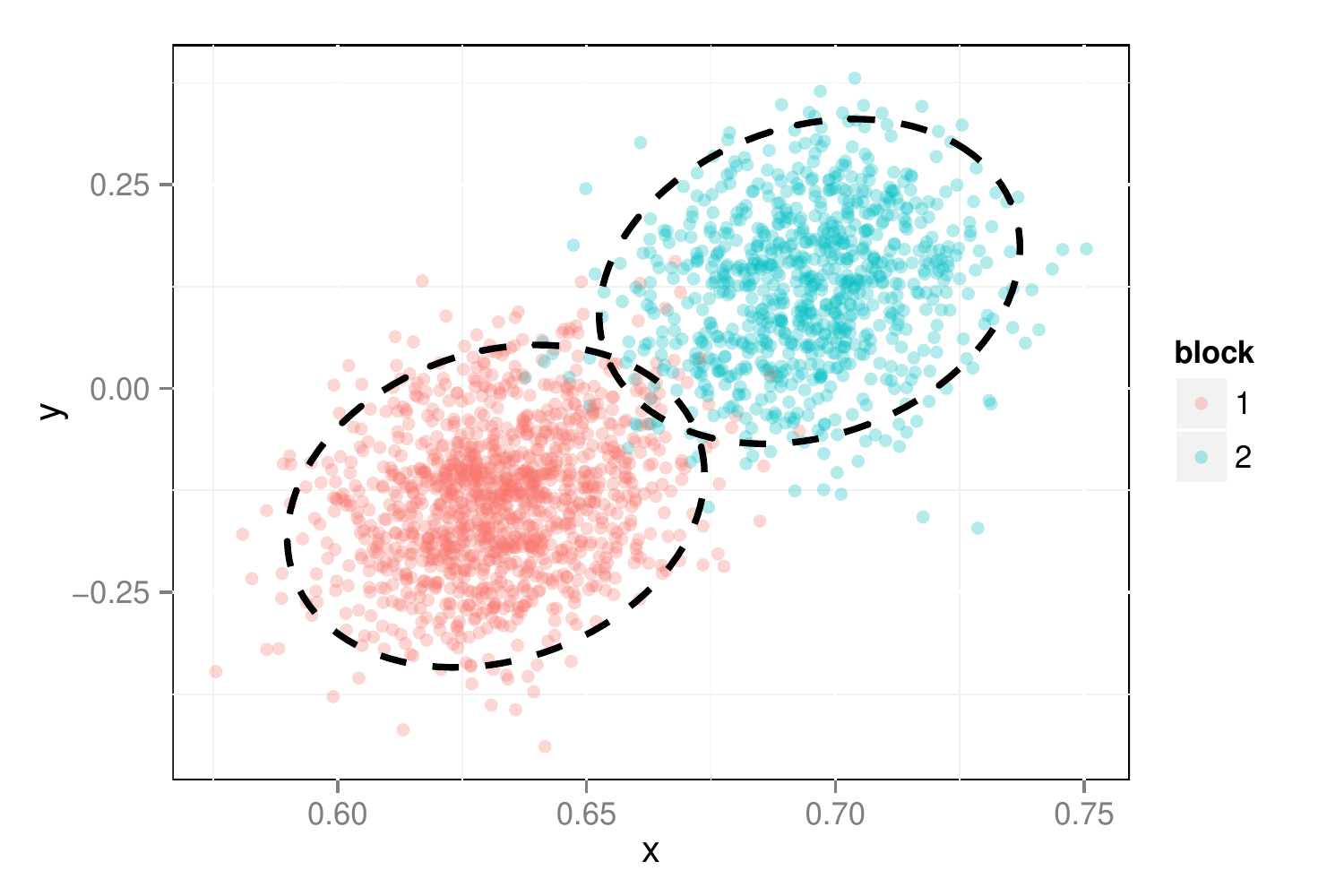}
  }
  \hfil
  \subfloat[$n = 4000$]{
    \includegraphics[width=5.5cm]{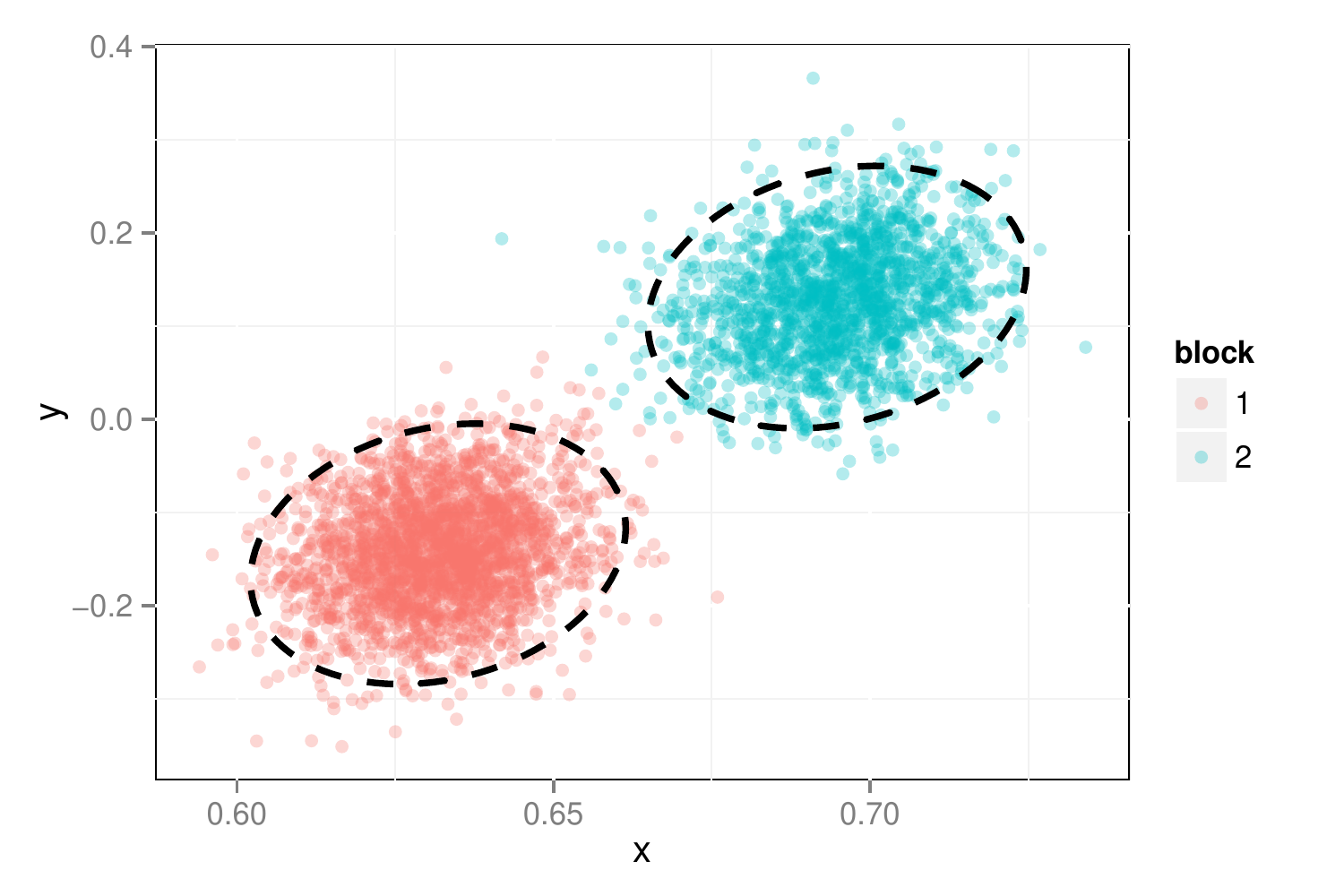}
    \label{fig5:subfig_scan}
  }
  \hfil
  \subfloat[$n = 8000$]{
    \includegraphics[width=5.5cm]{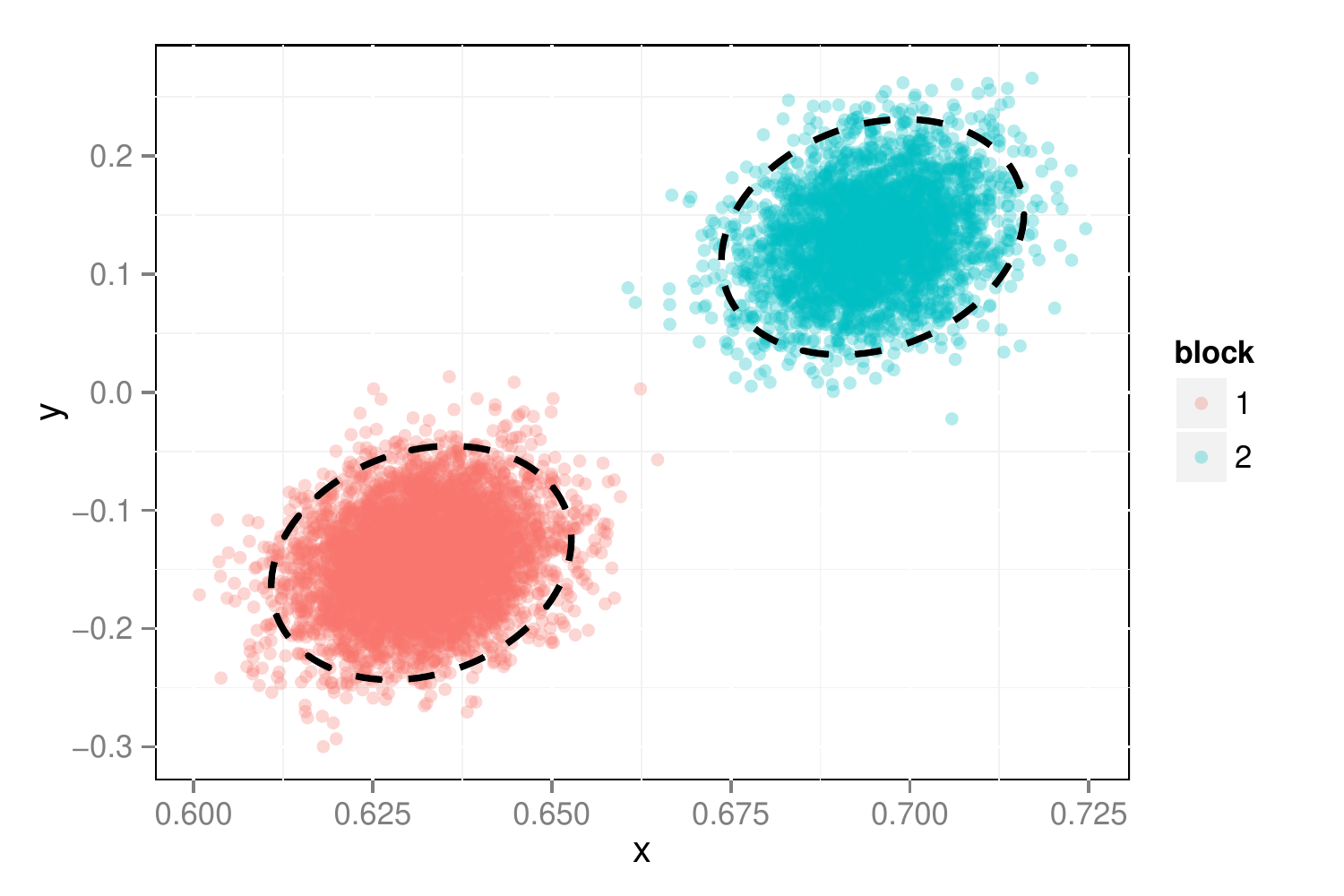}
  }
  \caption{Plot of the estimated latent positions  for $n
    \in \{1000,2000,4000,8000\}$. Dashed ellipses give the 95\% level
    curves for the distributions as specified in Theorem~\ref{thm:clt_multid}.}
  \label{fig:clusplot}
\end{figure*}
We then estimate the covariance matrices for the residuals. The
theoretical covariance matrices are given in the last line of
Table~\ref{tab:cov}, where $\Sigma_{1}$ and $\Sigma_{2}$ are the
covariance matrices for the residual $\sqrt{n}(\hat{X}_i - X_i)$ when
$X_i$ is from the first block and second block, respectively. The
empirical covariance matrices, denoted $\hat{\Sigma}_1$ and
$\hat{\Sigma}_2$, are computed by evaluating the sample covariance of
the rows of $\sqrt{n}\hat{X}_i$ corresponding to vertices in block 1
and 2 respectively.  The estimates of the covariance matrices are
given in Table~\ref{tab:cov}.  We see that as $n$ increases, the
sample covariances tend toward the specified limiting covariance
matrix given in the last row.

\begin{table}[htbp]
  \footnotesize
\begin{center}
\begin{tabular}{cccccc}
  $n$ &  2000 & 4000 & 8000 & 16000 & $\infty$  \\ \\ \midrule
$\hat{\Sigma}_1$ &%
$\begin{bmatrix} 0.58 & 0.54 \\ 0.54 & 16.56 \end{bmatrix}$ &%
$\begin{bmatrix} 0.58 & 0.63 \\ 0.63 & 14.87 \end{bmatrix}$ &%
$\begin{bmatrix} 0.60 & 0.61 \\ 0.61 & 14.20 \end{bmatrix}$ &%
$\begin{bmatrix} 0.59 & 0.58 \\ 0.58 & 13.96 \end{bmatrix}$ &%
$\begin{bmatrix} 0.59 & 0.55 \\ 0.55 & 13.07 \end{bmatrix}$ \\
\midrule
$\hat{\Sigma}_2$ &%
$\begin{bmatrix} 0.58 & 0.75 \\ 0.75 & 16.28 \end{bmatrix}$ &%
$\begin{bmatrix} 0.59 & 0.71 \\ 0.71 & 15.79 \end{bmatrix}$ &%
$\begin{bmatrix} 0.58 & 0.54 \\ 0.54 & 14.23 \end{bmatrix}$ &%
$\begin{bmatrix} 0.61 & 0.69 \\ 0.69 & 13.92 \end{bmatrix}$ &%
$\begin{bmatrix} 0.60 & 0.59 \\ 0.59 & 13.26 \end{bmatrix}$ \\ \midrule
\end{tabular}
\end{center}
\caption{The sample covariance matrices for $\sqrt{n}(\hat{X}_i-X_i)$
  for each block in a stochastic blockmodel with two blocks. Here
  $n \in \{2000,4000,8000,16000\}$. The
  last column are the theoretical covariance matrices for the limiting
  distribution.}
\label{tab:cov}
\end{table}


We also investigate the effects of the multivariate normal distribution
as specified in Theorem~\ref{thm:clt_multid} on inference procedures. It is shown in
\citet{STFP-2011,sussman2012universally} that the approach of
embedding a graph into some Euclidean space, followed by inference
(for example, clustering or classification) in that space can be
consistent. However, these consistency results are, in a sense, only
first-order results. In particular, they demonstrate only that the
error of the inference procedure converges to $0$ as the number of
vertices in the graph increases. We now illustrate how
Theorem~\ref{thm:clt_multid} may lead to a more refined error analysis.

We construct a sequence of random graphs on $n$ vertices, where $n$
ranges from $1000$ through $4000$ in increments of $250$, following
the stochastic blockmodel with parameters as given above in
Eq.~\eqref{eq:1}. For each graph $G_n$ on $n$ vertices, we embed $G_n$
and cluster the embedded vertices of $G_n$ via Gaussian mixture
model and K-Means. Gaussian mixture model-based clustering
was done using the MCLUST
implementation of \citep{fraley99:_mclus}.
We then measure the classification error of the
clustering solution. We repeat this procedure 100 times to obtain an
estimate of the misclassification rate. The results are plotted in
Figure~\ref{fig:gmm_kmeans_bayes}. For comparison, we also plot the
Bayes optimal classification error rate under the assumption that the
embedded points do indeed follow a multivariate normal mixture with
covariance matrices $\Sigma_1$ and $\Sigma_2$ as given above in the
last line of Table~\ref{tab:cov}. We also plot the misclassification
rate of $(C \log{n})/n$ as given in \citet{STFP-2011}
where the constant $C$ was chosen to match the misclassification rate
of $K$-means clustering for $n = 1000$. For the number of
vertices considered here, the upper bound for
the constant $C$ from \citet{STFP-2011} will give a vacuous upper
bound of the order of $10^6$ for the misclassification rate in this
example.
\begin{figure}[htbp]
  \centering
  \includegraphics[width=0.7\textwidth]{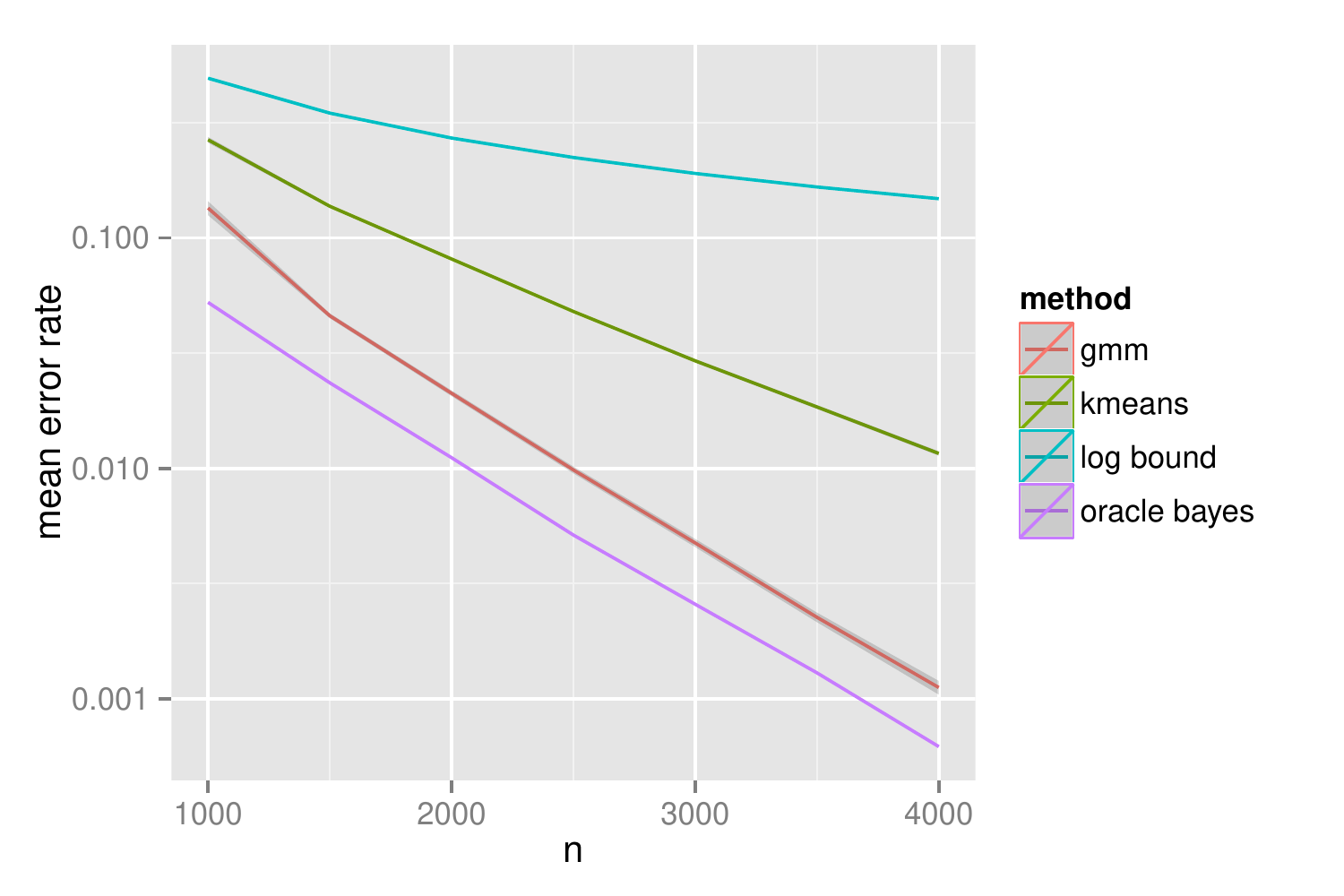}
  \caption{Comparison of classification error for Gaussian mixture model, K-Means,
    Bayes optimal error rate. The classification errors for each $n \in
    \{1000,1250,1500, \dots, 4000\}$ were obtained by averaging 100
    Monte Carlo iterations and are plotted on a $\log_{10}$ scale. The plot indicates that the assumption of mixture of
    multivariate normals can yield non-negligible
    improvement in the inference procedure. The log-bound curve shows an upper bound on the error rate as derived in \citet{STFP-2011}.}
  \label{fig:gmm_kmeans_bayes}
\end{figure}

\section{Discussion}
\label{sec:disc}
Our demonstration of the clustering accuracy in \S~\ref{sec:MultiD} shows how our
Theorem~\ref{thm:clt_multid} may impact statistical inference for random graphs. First,
we see that the empirical error rates are much lower than those proved
in previous work on spectral methods
\citep{rohe2011spectral,STFP-2011,fishkind2013consistent}. Indeed,
for both the $K$-Means algorithm and Gaussian mixture model, the average clustering
error decreases at an exponential rate as opposed to the $\log(n)/n$
bounds shown in previous work. Furthermore, the rate of decrease
for Gaussian mixture model-based
clustering closely mirrors the Bayes optimal error rate that
would be achieved if the estimated latent positions were exactly
distributed according to the multivariate normal distribution and the
parameters of this distribution were known.

These results suggest that further investigations using our theorem
could lead to much more accurate bounds on the empirical error rates
for adjacency spectral clustering. We believe that extending
Corollary~\ref{cor:cltK_multid} to the case in which $K$ is growing
with $n$, e.g., $K = n$, and
further work regarding distributions of spectral statistics for
stochastic blockmodels will lead to foundational statistical
procedures analogous to the results on estimation, hypothesis testing,
and clustering in the setting of mixtures of normal distributions in
Euclidean space. The relatively simple nature of
our spectral procedure allows for computationally efficient
statistical methodology.

Extensions of this work to a wider class of exchangeable graphs are
also of interest. Though not all exchangeable random graphs can be
represented as random dot product graphs, random dot product graphs
can approximate any exchangeable graph in the following sense: given a
sufficiently regular link function, there exists a {\em feature map}
from the original latent position space to $\ell_2$, such that the
link function applied to the original latent positions is equal to the
inner product applied to the feature-mapped positions in
$\ell_2$. \citet{tang2012universally} argue that by increasing the
dimension of the estimated latent positions, it is possible to
estimate these feature-mapped latent positions in a way that allows
for consistent subsequent inference. Though this larger class of
models is not considered here, we believe this is strong motivation to
study the random dot product graph model and its eigenvalues and
eigenvectors.

\bibliographystyle{plainnat}
\bibliography{EvecCLT_MultiD_oned_separate}

\end{document}